\pdfoutput=1
\RequirePackage{ifpdf}
\ifpdf 
\documentclass[pdftex]{sigma}
\else
\documentclass{sigma}
\fi

\usepackage{tikz}
\usetikzlibrary{calc}

\numberwithin{equation}{section}

\newtheorem{Theorem}{Theorem}[section]
\newtheorem{Corollary}[Theorem]{Corollary}
\newtheorem{Lemma}[Theorem]{Lemma}
\newtheorem{Proposition}[Theorem]{Proposition}
{ \theoremstyle{definition}
\newtheorem{Definition}[Theorem]{Definition}
\newtheorem{Example}[Theorem]{Example}
\newtheorem{Remark}[Theorem]{Remark} }

\newcommand{\IC}{\mathbb{C}}
\newcommand{\IR}{\mathbb{R}}
\newcommand{\XX}{\mathfrak{X}}

\def \g {\mathfrak{a}}

\def \g {\mathfrak{g}}
\def \h {\mathfrak{h}}

\def \A {{\mathcal{A}}}

\def \P {{\mathcal{P}}}
\def \Q {{\mathcal{Q}}}

\DeclareMathOperator{\Ad}{Ad}
\DeclareMathOperator{\Hom}{Hom}
\DeclareMathOperator{\diag}{diag}
\DeclareMathOperator{\Mix}{Mix}
\DeclareMathOperator{\Fus}{Fus}
\def \sV {{\scriptscriptstyle V}}
\def \sY {{\scriptscriptstyle Y}}
\def \sGam {{\scriptscriptstyle \Gamma}}
\newcommand{\calpha}{\check{\alpha}}
\newcommand{\la}{\langle}
\newcommand{\ra}{\rangle}
\newcommand{\del}{\partial}
\newcommand{\piY}{{\pi_{\scriptscriptstyle Y}}}

\begin{document}

\allowdisplaybreaks

\newcommand{\arXivNumber}{1610.01782}

\renewcommand{\PaperNumber}{063}

\FirstPageHeading

\ShortArticleName{The Fock--Rosly Poisson Structure as Def\/ined by a Quasi-Triangular $r$-Matrix}

\ArticleName{The Fock--Rosly Poisson Structure\\ as Def\/ined by a Quasi-Triangular $\boldsymbol{r}$-Matrix}

\Author{Victor MOUQUIN}

\AuthorNameForHeading{V.~Mouquin}

\Address{University of Toronto, Toronto ON, Canada}
\Email{\href{mailto:mouquinv@math.toronto.edu}{mouquinv@math.toronto.edu}}

\ArticleDates{Received March 26, 2017, in f\/inal form August 01, 2017; Published online August 09, 2017}

\Abstract{We reformulate the Poisson structure discovered by Fock and Rosly on moduli spaces of f\/lat connections over marked surfaces in the framework of Poisson structures def\/ined by Lie algebra actions and quasitriangular $r$-matrices, and we show that it is an example of a mixed product Poisson structure associated to pairs of Poisson actions, which were studied by J.-H.~Lu and the author. The Fock--Rosly Poisson structure corresponds to the quasi-Poisson structure studied by Massuyeau, Turaev, Li-Bland, and \v{S}evera under an equivalence of categories between Poisson and quasi-Poisson spaces.}


\Classification{53D17; 53D30; 17B62}

\section{Introduction}\label{intro}

Let $G$ be a connected complex Lie group with Lie algebra $\g$, and let $s \in S^2\g$ be a $\g$-invariant element. The moduli space of f\/lat $G$-connections on a Riemann surface $\Sigma$ has the well known~\cite{AB} canonical Atiyah--Bott Poisson structure. If one ``marks'' f\/initely many points $V \subset \del \Sigma$ in the boundary of $\Sigma$ and consider only gauge transformations which are trivial over $V$, Fock and Rosly have constructed in \cite{FR2, FR1} a Poisson structure $\pi_r$ on the corresponding moduli space~$\A(\Sigma, V)$, which depends on a quasitriangular $r$-matrix $r_v$ for every $v \in V$ such that all $r_v$'s have symmetric part~$s$. Under the quotient by the group of lattice gauge transformations~$G^V$, $\pi_r$ descends to the Atiyah--Bott Poisson structure on the full moduli space, and quantizations of~$\pi_r$ play a~fundamental role in quantum gravity (see~\cite{physics1} and references therein).

The bivector f\/ield $\pi_r$ was given in~\cite{FR2, FR1} by a formula, of which the proof that it def\/ines a~Poisson structure was left as a computation. In this paper, as an application of the methods developed in \cite{Lu-Mou:mixed}, we give a simpler and more conceptual proof that $\pi_r$ is a Poisson structure, by viewing it in the framework of Poisson structures {\it defined by a Lie algebra action and a~quasitriangular $r$-matrix}. Recall that given an action $\rho\colon \h \to \XX^1(Y)$ of a Lie algebra $\h$ on a~mani\-fold~$Y$ and a~quasitriangular $r$-matrix $r \in \h \otimes \h$, if the pushforward $\piY= \rho(r)$ is a~bivector f\/ield, it is automatically Poisson, and one says that $\piY$ is def\/ined by the action $\rho$ and the $r$-mat\-rix~$r$.

More precisely, given an oriented skeleton $\Gamma$ of a marked surface $(\Sigma, V)$, one has a natural action $\sigma_\sGam$ of the Lie algebra $\g^{\sGam_{1/2}}$ on $\A(\Sigma, V)$, where $\Gamma_{1/2}$ is the set of half edges of $\Gamma$, and a~quasitriangular $r$-matrix $r_\sGam \in \g^{\sGam_{1/2}} \otimes \g^{\sGam_{1/2}}$, such that $\sigma_\sGam(r_\sGam)$ is a~Poisson structure. Both $\sigma_\sGam$ and~$r_\sGam$ depend on~$\Gamma$, but one proves that $\pi_r = \sigma_\sGam(r_\sGam)$ does not.

Marked surfaces can be {\it fused} at their marked points. One also has the notion introduced in~\cite{Lu-Mou:mixed} of {\it fusion of Poisson spaces} admitting a~Poisson action by a quasitriangular Lie bialgebra, and we show that the Poisson structures on the associated moduli spaces correspond under these constructions. In particular, the Fock--Rosly Poisson structure is an example of a {\it mixed product Poisson structure associated to pairs of Poisson actions} introduced in~\cite{Lu-Mou:mixed}.

On the other hand, $\A(\Sigma, V)$ carries a canonical {\it quasi-Poisson structure} $Q_s$, f\/irst discovered in \cite{mass-turaev} when $V$ is a singleton, and further studied in \cite{LBS1} for general $V$'s, which can be obtained by reduction of the canonical symplectic structure on the inf\/inite-dimensional af\/f\/ine space of $G$-connections on $\Sigma$. Quasi-Poisson manifolds were introduced in \cite{Anton-Yvette, AA} as a way to obtain a~unif\/ied picture of various notions of moment maps. It is shown in \cite{Anton-Yvette, David-Severa:quasi-Hamiltonian-groupoids, Lu-Mou:mixed} (see also Section~\ref{subsec-quasi}) that one has an equivalence of categories between the category of $(\g, \phi_s)$ quasi-Poisson spaces and the category of $(\g, r)$ Poisson spaces, where $r$ is a quasitriangular $r$-matrix whose symmetric part is $s$, and $\phi_s \in \wedge^3 \g$ is the Cartan element associated to $s$ (see \eqref{eq-cyb}). We show in this paper that $\pi_r$ corresponds to $Q_s$ under this equivalence of categories.

An interesting project would be to develop a theory of quantizations of Poisson structures def\/ined by actions of Lie algebras and quasitriangular $r$-matrices. This paper provides the setting to study the quantization of the Fock--Rosly Poisson structure from this point of view.

The paper is organized as follows. In Section~\ref{sec-quasitri} we recall the basic facts on quasitriangular $r$-matrices which will be needed later, and in Section~\ref{sec-graphs&surfaces} we recall the fusion of ciliated graphs and marked surfaces. The Poisson structure $\pi_r$ on the moduli space $\A(\Sigma, V)$ is def\/ined in Section~\ref{sec-FR-Poiss}, where we prove that it is independent of the choice of an oriented skeleton of $(\Sigma, V)$, and that fusion of marked surfaces corresponds to fusion of the associated Poisson structures. In Section~\ref{sec-quasipoisson}, the equivalence between $\pi_r$ and the quasi-Poisson structure $Q_s$ under an equivalence of categories between Poisson and quasi-Poisson spaces is proven.

\subsection{Notation} \label{nota}

Throughout this paper, vector spaces are understood to be over $\IR$ or $\IC$.

If $\Gamma$ is a f\/inite set and $\{X_\gamma \colon \gamma \in \Gamma\}$ a family of sets indexed by $\Gamma$, for $x \in \prod_{\gamma \in \Gamma} X_\gamma$ and $\gamma \in \Gamma$, $x_\gamma \in X_\gamma$ denotes the $\gamma$-component of $x$. If $\{V_\gamma\colon \gamma \in \Gamma\}$ is a family of groups and $v \in V_\gamma$, $(v)_\gamma \in \bigoplus_{\gamma \in \Gamma}V_\gamma$ is the image of $v$ under the embedding of $V_\gamma$ into $\bigoplus_{\gamma \in \Gamma}V_\gamma$ as the $\gamma$-component. When the $V_\gamma$'s are vector spaces, we extend this notation to tensor powers. Namely, for an integer $k \geq 1$ and $v \in V_\gamma^{\otimes k}$, $(v)_\gamma$ is the image of $v$ under the embedding of $V_\gamma^{\otimes k}$ into $(\bigoplus_{\gamma \in \Gamma}V_\gamma)^{\otimes k}$ as the $\gamma$-component.

If $\rho\colon Y \times G \to Y$ (resp.\ $\lambda\colon G \times Y \to Y$) is a right (resp.\ left) action of a Lie group $G$ on a~mani\-fold~$Y$, we will denote by $\rho\colon \g \to \XX^1(Y)$ (resp.\ $\lambda\colon \g \to \XX^1(Y)$) the induced right (resp.\ left) Lie algebra action of the Lie algebra~$\g$ of $G$ on $Y$. If $x \in \g^{\otimes k}$, $k \geq 1$, we denote respectively by~$x^R$ and~$x^L$ the right and left invariant $k$-tensor f\/ield on $G$ whose value at the identity $e \in G$ is~$x$.

Lie bialgebras will be denoted as pairs $(\g, \delta_\g)$, where $\g$ is a Lie algebra, and $\delta_\g\colon \g \to \wedge^2 \g$ the cocycle map. Recall that $\delta_\g$ satisf\/ies
\begin{gather*}
\delta_\g([x,y]) = [x, \delta_\g(y)] + [\delta_\g(x), y], \qquad x,y \in \g,
\end{gather*}
and that the dual map $\delta_\g^*\colon {\wedge}^2 \g^* \to \g^*$ is a Lie bracket on $\g^*$.

\section[Poisson structures def\/ined by quasitriangular $r$-matrices]{Poisson structures def\/ined by quasitriangular $\boldsymbol{r}$-matrices} \label{sec-quasitri}

We recall in this section basic facts about quasitriangular $r$-matrices and refer to \cite{Lu-Mou:mixed} for a~detailed exposition on Poisson Lie groups and Lie bialgebras.

\subsection[Quasitriangular $r$-matrices]{Quasitriangular $\boldsymbol{r}$-matrices}

Let $\g$ be a f\/inite-dimensional Lie algebra, and let $r = s + \Lambda \in \g \otimes \g$, with $s \in (S^2\g)^\g$ and $\Lambda \in \wedge^2 \g$. One says that $r$ is a {\it quasitriangular $r$-matrix} on $\g$ if it satisf\/ies the {\it classical Yang--Baxter equation}
\begin{gather*}
[\Lambda, \Lambda] + \phi_s = 0, 
\end{gather*}
where $[\, , \,]\colon \wedge^k \g \otimes \wedge^l \g \to \wedge^{k+l-1} \g$ is the Schouten bracket on the exterior powers of a Lie algebra, and $\phi_s \in \wedge^3 \g$ is def\/ined by
\begin{gather}
\phi_s(\xi, \eta, \zeta) = 2 \la \xi, [s^\#(\eta), \, s^\#(\zeta)]\ra, \qquad \xi, \eta, \zeta \in \g^*, \label{eq-cyb}
\end{gather}
where $s^\sharp\colon \g^* \to \g$ is given by $\la s^\sharp(\xi), \eta \ra = s(\xi, \eta)$, $\xi, \eta \in \g^*$. If $r = \sum_i x_i \otimes y_i \in \g \otimes \g$ is a~quasitriangular $r$-matrix, it def\/ines a Lie bialgebra structure
\begin{gather*}
\delta_r\colon \ \g \to \g \wedge \g, \qquad \delta_r(x) = \sum_i [x, x_i] \otimes y_i + x_i \otimes [x, y_i], \qquad x \in \g,
\end{gather*}
and one calls the pair $(\g, r)$ a {\it quasitriangular Lie bialgebra}.

Let $(\g, \delta_\g)$ be a Lie bialgebra and $(Y, \piY)$ a Poisson manifold. A (right) Poisson action of~$(\g, \delta_\g)$ on~$(Y, \piY)$ is a Lie algebra morphism $\rho\colon \g \to \XX^1(Y)$ satisfying
\begin{gather*}
[\rho(x), \piY] = \rho(\delta_\g(x)), \qquad x \in \g,
\end{gather*}
and one also says that $(Y, \piY, \rho)$ is a right $(\g, \delta_\g)$-Poisson space.

Let $\g$ be a f\/inite-dimensional Lie algebra, $Y$ a manifold and $\rho\colon \g \to \XX^1(Y)$ a right action of~$\g$ on~$Y$. For $r = \sum_i x_i \otimes y_i \in \g \otimes \g$, one has the 2-tensor f\/ield
\begin{gather*}
\rho(r) = \sum_i \rho(x_i) \otimes \rho(y_i) \in \XX^1(Y) \otimes \XX^1(Y),
\end{gather*}
and writing $r = s + \Lambda$, with $s \in S^2 \g$ and $\Lambda \in \wedge^2 \g$, it is clear that $\rho(r)$ is a bivector f\/ield on~$Y$ if and only if $\rho(s) = 0$.

\begin{Proposition}[{\cite[Proposition 2.18]{Lu-Mou:mixed}}]\label{Poi-def-r}
If $r$ is a quasitriangular $r$-matrix and if $\rho(r)$ is a~bivector field, it is a Poisson bivector field, and $(Y, \rho(r), \rho)$ is a right $(\g, r)$-Poisson space.
\end{Proposition}

In the context of Proposition \ref{Poi-def-r}, one says that $\rho(r)$ is a Poisson structure {\it def\/ined by the quasitriangular $r$-matrix $r$ and the action $\rho$}.

Let $\g$ be a Lie algebra and $n \geq 1$ an integer. For any $r = \sum_i x_i \otimes y_i \in \g \otimes \g$, def\/ine $\Mix^n(r) \in \wedge^2 (\g^n)$ by
\begin{gather}
\Mix^n(r) = \sum_{1 \leq j < k \leq n} \big( \Mix^n(r) \big)_{j,k}, \label{Mix^n(r)}
\end{gather}
where
\begin{gather*}
\big( \Mix^n(r) \big)_{j,k} = \sum_i (y_i)_j \wedge (x_i)_k, \qquad 1 \leq j < k \leq n,
\end{gather*}
and for any sign function $\varepsilon\colon \{1, \ldots, n\} \to \{1, -1\}$, let
\begin{gather*}
r^{\varepsilon, n} = (\varepsilon(1)s, \ldots, \varepsilon(n)s) + (\Lambda, \ldots, \Lambda) \in \g^n \otimes \g^n,
\end{gather*}
where $r = s + \Lambda$ with $s \in S^2\g$ and $\Lambda \in \wedge^2 \g$, and let
\begin{gather}
r^{(\varepsilon, n)} = r^{\varepsilon, n} - \Mix^n(r) \in \g^n \otimes \g^n. \label{defn-r^(n)}
\end{gather}

\begin{Theorem}[{\cite[Theorem 6.2]{Lu-Mou:mixed}}] \label{r^(n)}
If $r \in \g \otimes \g$ is a quasitriangular $r$-matrix on $\g$, then for any $n \geq 1$ and any sign function $\varepsilon$, $r^{(\varepsilon, n)}$ is a quasitriangular $r$-matrix on~$\g^n$, and the Lie bialgebra structure
\begin{gather}
\delta_r^{(n)} = \delta_{r^{(\varepsilon, n)}} \label{delta_r^n}
\end{gather}
is independent of $\varepsilon$. Moreover, the map
\begin{gather*}
\diag_n\colon \ (\g, \delta_r) \to \big(\g^n, \delta_r^{(n)}\big), \qquad \diag_n(x) = (x, \ldots, x), \qquad x \in \g, 
\end{gather*}
is an embedding of Lie bialgebras.
\end{Theorem}

For any $r \in \g \otimes \g$ and any sign function $\varepsilon$, denote by $\Lambda_r^{(n)}$ the anti-symmetric part of $r^{(\varepsilon, n)}$. Writing $r = \sum_i x_i \otimes y_i = s + \Lambda$, with $2s = \sum_i (x_i \otimes y_i + y_i \otimes x_i)$ and $2\Lambda = \sum_i x_i \wedge y_i$, one has explicitly
\begin{gather*}
\Lambda_r^{(n)} = (\Lambda, \ldots, \Lambda) - \Mix^n(r) = \frac{1}{2} \sum_{j=1}^n \sum_i (x_i)_j \wedge (y_i)_j - \sum_{1 \leq j < k \leq n} \sum_i (y_i)_j \wedge (x_i)_k. 
\end{gather*}

The following lemma will be used in the proof of Proposition \ref{coord-quasi}.

\begin{Lemma} \label{lem-diag_nLambda}
Let $r = s + \Lambda \in \g \otimes \g$, with $s \in S^2 \g$ and $\Lambda \in \wedge^2 \g$. Then
\begin{gather*}
\Lambda_r^{(n)} - \diag_n(\Lambda) = - \Mix^n(s). 
\end{gather*}
\end{Lemma}
\begin{proof}
A straightforward calculation shows that $\diag_n(\Lambda) = (\Lambda, \ldots, \Lambda) - \Mix^n(\Lambda)$. Thus
\begin{gather*}
\Lambda_r^{(n)} - \diag_n(\Lambda) = - \Mix^n(r) + \Mix^n(\Lambda) = - \Mix^n(s).\tag*{\qed}
\end{gather*}\renewcommand{\qed}{}
\end{proof}

The following lemma will be used in the proof of Theorem \ref{thm-fuse-mark}.

\begin{Lemma} \label{lem-mix^2}
Let $r\in \g \otimes \g$. For integers $m,n \geq 0$, one has
\begin{gather*}
\big(\Lambda_r^{(m)}, \Lambda_r^{(n)}\big) - (\diag_m, \diag_n)\big(\Mix^2(r)\big) = \Lambda_r^{(m+n)} \in \wedge^2 \big(\g^{m+n}\big).
\end{gather*}
\end{Lemma}
\begin{proof}
Indeed, writing $r = \sum_i x_i \otimes y_i$ and letting $\Lambda \in \wedge^2 \g$ be the anti-symmetric part of $r$, one has
\begin{gather*}
(\diag_m, \diag_n)\big(\Mix^2(r)\big) = \sum_{1 \leq k \leq m, \, 1 \leq l \leq n} (y_i)_k \wedge (x_i)_l,
\end{gather*}
hence
\begin{gather*}
\big(\Lambda_r^{(m)}, \Lambda_r^{(n)}\big) - (\diag_m, \diag_n)\big(\Mix^2(r)\big) \\
\qquad{} = (\Lambda, \ldots, \Lambda) - \sum_{1 \leq k < l \leq m+n} (y_i)_k \wedge (x_i)_l = \Lambda_r^{(m+n)}.\tag*{\qed}
\end{gather*}\renewcommand{\qed}{}
\end{proof}

\subsection{Fusion products of Poisson spaces} \label{subsec-fusion}

Let $n \geq 1$ be an integer, $r \in \g \otimes \g$ a quasitriangular $r$-matrix on a Lie algebra $\g$, and let $(Y, \piY)$ be a~Poisson manifold with a right Poisson action $\rho\colon \g^n \to \XX^1(Y)$ of $(\g, r)^n$, and a~right Poisson action $\psi\colon \h \to \XX^1(Y)$ of a Lie bialgebra $(\h, \delta_\h)$, so that $(Y, \piY, \rho \times \psi)$ is a right $(\g, r)^n \times (\h, \delta_\h)$-Poisson space. By \cite[Lemma~2.16]{Lu-Mou:mixed} and Theorem~\ref{r^(n)}, the triple
\begin{gather}
\Fus_{(\g, r)^n}(Y, \piY, \rho \times \psi) := \big(Y, \piY - \rho\big(\Mix^n(r)\big), (\rho \circ \diag_n) \times \psi \big) \label{eq-fusion}
\end{gather}
is a right $(\g, r) \times (\h, \delta_\h)$-Poisson space, which we call the {\it fusion at $(\g, r)^n$ of $(Y, \piY, \rho \times \psi)$}. As a particular case, suppose that $\h=0$, that
\begin{gather*}
(Y_1, \pi_{\sY_1}, \rho_1), \ldots, (Y_n, \pi_{\sY_n}, \rho_n)
\end{gather*}
are right $(\g, r)$-Poisson spaces, that $Y = Y_1 \times \cdots \times Y_n$ is equipped with the direct product Poisson structure $\piY = \pi_{\sY_1} \times \cdots \times \pi_{\sY_n}$, and that $\rho\colon \g^n \to \XX^1(Y)$ is given by
\begin{gather*}
\rho(x_1, \ldots, x_n) = (\rho_1(x_1), \ldots, \rho_n(x_n)), \qquad x_j \in \g.
\end{gather*}
The $(\g, r)$-Poisson space in \eqref{eq-fusion} is called in \cite{Lu-Mou:mixed} the {\it fusion product of $(Y_j, \pi_{\sY_j}, \rho_j)$, $1 \leq j \leq n$}.

\section{Ciliated graphs and marked surfaces} \label{sec-graphs&surfaces}

In this section, we review the fusion of marked surfaces and ciliated graphs. Our main references are \cite{FR1, LBS1}.

\subsection{Ciliated graphs and marked surfaces} \label{subsec-ciliated}

A {\it marked surface} $(\Sigma, V)$ is a compact Riemann surface, together with a non-empty f\/inite collection of points $V \subset \del \Sigma$ lying in the boundary of $\Sigma$. A {\it skeleton} of a marked surface $(\Sigma, V)$ is a graph $\Gamma$ embedded in $\Sigma$, with set of vertices $V$ and such that $\Sigma$ deformation retracts onto $\Gamma$.

\begin{Proposition}[{\cite[Section 4]{LBS1}}]\label{ex-unique}
Any marked surface $(\Sigma, V)$ admits a skeleton, and skeletons of $(\Sigma, V)$ are unique up to isomorphisms and local changes
 \begin{gather} \label{loc-move}\begin{split}&
\begin{tikzpicture}
\begin{scope}[shift={(-2,0)}]
\draw [fill=black] (0,0) circle [radius=0.05];
\draw [fill=black] (1,0.5) circle [radius=0.05];
\draw [fill=black] (0,1) circle [radius=0.05];
\draw [black] (0,0) to (1,0.5);
\draw [black] (1,0.5) to (0,1);
\end{scope}
\node at (0.5,0.5) {$\longleftrightarrow$};
\begin{scope}[shift={(2,0)}]
\draw [fill=black] (0,0) circle [radius=0.05];
\draw [fill=black] (1,0.5) circle [radius=0.05];
\draw [fill=black] (0,1) circle [radius=0.05];
\draw [black] (0,0) to (1,0.5);
\draw [black] (0,0) to (0,1);
\end{scope}
\end{tikzpicture}
\end{split}
\end{gather}
\end{Proposition}

Let $\Gamma$ be a skeleton of a marked surface $(\Sigma, V)$. For every $v \in V$, the orientation of $\Sigma$ induces a linear ordering of the half edges incident to $v$, thus one is led to formulate the following

\begin{Definition}[\cite{FR1, LBS1}]
A {\it ciliated graph} is a graph in which each vertex is equipped with a~linear order of the half edges incident to it.
\end{Definition}

The name is inspired by the fact that when drawing a ciliated graph, one can specify the linear order of half edges at each vertex by drawing a small ``cilium" between the minimal and maximal half edge.

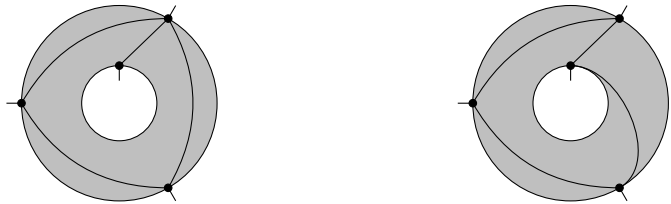
\begin{figure}[h]\centering
\begin{tikzpicture}
\pgfmathsetmacro{\ra}{0.5}
\pgfmathsetmacro{\rb}{1.3}
\begin{scope}[shift={(-3,0)}]
\draw [fill=lightgray] (0,0) circle [radius=\rb];
\draw [fill=white] (0,0) circle [radius=\ra];
\draw [fill=black] (0,\ra) circle [radius=0.05];
\draw [fill=black] (180:\rb) circle [radius=0.05];
\draw [fill=black] (60:\rb) circle [radius=0.05];
\draw [fill=black] (-60:\rb) circle [radius=0.05];
\draw [black] (180:\rb) to [out=60, in=180] (60:\rb);
\draw [black] (60:\rb) to [out=-60, in=60] (-60:\rb);
\draw [black] (180:\rb) to [out=-60, in=180] (-60:\rb);
\draw [black] (0,\ra) to (60:\rb);
\draw [black] (180:\rb) to (180:\rb+0.2);
\draw [black] (0,\ra) to (0,\ra-0.2);
\draw [black] (60:\rb) to (60:\rb+0.2);
\draw [black] (-60:\rb) to (-60: \rb+0.2);
\end{scope}
\begin{scope}[shift={(3,0)}]
\draw [fill=lightgray] (0,0) circle [radius=\rb];
\draw [fill=white] (0,0) circle [radius=\ra];
\draw [fill=black] (0,\ra) circle [radius=0.05];
\draw [fill=black] (180:\rb) circle [radius=0.05];
\draw [fill=black] (60:\rb) circle [radius=0.05];
\draw [fill=black] (-60:\rb) circle [radius=0.05];
\draw [black] (180:\rb) to [out=60, in=180] (60:\rb);
\draw [black] (180:\rb) to [out=-60, in=180] (-60:\rb);
\draw [black] (0,\ra) to (60:\rb);
\draw [black] (0,\ra) to [out=0, in=30] (-60:\rb);
\draw [black] (180:\rb) to (180:\rb+0.2);
\draw [black] (0,\ra) to (0,\ra-0.2);
\draw [black] (60:\rb) to (60:\rb+0.2);
\draw [black] (-60:\rb) to (-60: \rb+0.2);
\end{scope}
\end{tikzpicture}
\caption{An annulus with four marked points and two non-isomorphic skeletons with their cilium structure.}
\end{figure}

We introduce further notations in order to discuss ciliated graphs. Let $\Gamma$ be a ciliated graph with set of vertices $V$ and set of edges $\Gamma_1$. Denote by $\Gamma_{1/2}$ the set of half edges of $\Gamma$, and note that $\Gamma_{1/2}$ comes with a natural involution with no f\/ixed points $\alpha \mapsto \calpha$, mapping a half edge to the opposite half edge, and for $\alpha \in \Gamma_{1/2}$ we write $[\alpha, \calpha]$ for the edge composed of the two half edges $\alpha$ and $\calpha$. For every $v \in V$, let $\Gamma_v \subset \Gamma_{1/2}$ be the set of half edges incident to $v$, so that $\Gamma_{1/2} = \bigsqcup_{v \in V} \Gamma_v$ and $\Gamma_v$ is a linearly ordered set for each $v \in V$.

\subsection{Fusion of ciliated graphs and marked surfaces}

We recall now the fusion of marked surfaces and ciliated graphs.

Let $(\Sigma, V)$ be a marked surface. Since $\Sigma$ is oriented, every $v \in V$ def\/ines a piece of arc in $\del \Sigma$ starting at $v$ and a piece of arc in $\del \Sigma$ ending at $v$. For a pair $(v_1, v_2)$ of distinct elements of~$V$, the {\it fusion of $\Sigma$ at $(v_1, v_2)$} is the marked surface $(\Sigma_{(v_1, v_2)}, V_{v_1= v_2})$ obtained by gluing a piece of arc ending in~$v_1$ with a piece or arc starting at $v_2$, so that $v_1$ and $v_2$ are identif\/ied. The set of marked points $V_{v_1=v_2}$ on $\Sigma_{(v_1, v_2)}$ is then obtained from~$V$ by identifying~$v_1$ and~$v_2$.

\begin{center}
\begin{tikzpicture}
\pgfmathsetmacro{\ra}{1}
\begin{scope}[shift={(-3,0)}]
\coordinate (o_1) at (-1.5, 0);
\coordinate (o_2) at (1.5, 0);
\path [fill=lightgray] (o_1) circle [radius=\ra];
\path [fill=white] ($(o_1) + (-\ra, -\ra)$) rectangle ($(o_1) + (0,\ra)$);
\path [fill=lightgray] (o_2) circle [radius=\ra];
\path [fill=white] ($(o_2) + (0, -\ra)$) rectangle ($(o_2)+ (\ra, \ra)$);
\coordinate (v_2) at ($(o_1) + (\ra, 0)$);
\coordinate (v_1) at ($(o_2) - (\ra, 0)$);
\draw [fill=black] (v_2) circle [radius=0.04];
\node [left] at (v_2) {$v_2$};
\draw [fill=black] (v_1) circle [radius=0.04];
\node [right] at (v_1) {$v_1$};
\draw [very thin] (o_1) ++ (0, \ra) arc (90:-90:\ra);
\draw [very thick] (v_2) arc (0:35:\ra);
\draw [very thin] ($(o_2)+ (0, -\ra)$) arc (270:90:\ra);
\draw [very thick] (v_1) arc (180:145:\ra);
\end{scope}
\node at (0,0) {$\stackrel{{\rm fusion}}{\longrightarrow}$};
\begin{scope}[shift={(3,0)}]
\coordinate (o_1) at (-1,0);
\coordinate (o_2) at (1,0);
\path [fill=lightgray] (o_1) circle [radius=\ra];
\path [fill=lightgray] (o_2) circle [radius=\ra];
\draw [very thin] ($(o_1) + (0, -\ra)$) arc (-90: 0:\ra);
\draw [very thin] ($(o_2) + (0, -\ra)$) arc (270: 180:\ra);
\path [fill=white] ($(o_1) - (\ra,\ra)$) rectangle ($(o_1)+(0,\ra)$);
\path [fill=white] ($(o_2) + (0,-\ra)$) rectangle ($(o_2)+(\ra,\ra)$);
\path [fill=white] (o_1) rectangle ($(o_2)+(\ra,\ra)$);
\path [fill=lightgray] ($(o_1)+(0,\ra)$) to [out=0,in=135] ($(o_1)+(\ra, 0.7)$) to [out=45,in=180] ($(o_2)+(0,\ra)$) to (o_2) to (o_1) to ($(o_1)+(0,\ra)$);
\draw [very thin] ($(o_1)+(0,\ra)$) to [out=0,in=135] ($(o_1)+(\ra, 0.7)$) to [out=45,in=180] ($(o_2)+(0,\ra)$);
\draw [fill=black] ($(o_1)+(\ra,0)$) circle [radius=0.04];
\node [above] at ($(o_1)+(\ra,0)$) {$v_1=v_2$};
\end{scope}
\end{tikzpicture}
\end{center}

Let $\Gamma$ be a ciliated graph with set of vertices $V$ and edges $\Gamma_1$, and let $(v_1, v_2)$ be a pair of distinct vertices, with $\Gamma_{v_1} = \{\alpha_1< \dots < \alpha_k\}$ and $\Gamma_{v_2} = \{\alpha_{k+1} < \dots < \alpha_l\}$. The {\it fusion of $\Gamma$ at $(v_1, v_2)$} is the ciliated graph $\Gamma_{(v_1, v_2)}$ obtained by identifying $v_1$ and $v_2$, and with linear order on the set $\Gamma_{v_1 = v_2}$ of half edges incident to $v_1 = v_2$ given by $\alpha_1 < \cdots < \alpha_k < \alpha_{k+1} < \dots < \alpha_l$.

Note that the fusion of marked surfaces and ciliated graphs are associative operations, but not necessarily commutative. The following lemma is straightforward.

\begin{Lemma} \label{fusion-graph-surf}
Let $(\Sigma, V)$ be a marked surface with skeleton $\Gamma$, and let $v_1, v_2 \in V$ be distinct points. Then the image of $\Gamma$ under the fusion map $(\Sigma, V) \to (\Sigma_{(v_1, v_2)}, V_{v_1 =v_2})$ is isomorphic to~$\Gamma_{(v_1, v_2)}$, and is a skeleton for $(\Sigma_{(v_1, v_2)}, V_{v_1= v_2})$.
\end{Lemma}

Since
\begin{tikzpicture}
\draw [fill=black] (0,0.6) circle [radius=0.05];
\draw [fill=black] (0.7,0.6) circle [radius=0.05];
\draw [black] (0,0.6) to (0.7,0.6);
\end{tikzpicture}
is a skeleton for a disk with two marked points, and since every ciliated graph can be obtained by successive fusion of copies of
\begin{tikzpicture}
\draw [fill=black] (0,0.6) circle [radius=0.05];
\draw [fill=black] (0.7,0.6) circle [radius=0.05];
\draw [black] (0,0.6) to (0.7,0.6);
\end{tikzpicture},
 every marked surface can be obtained by successive fusion of disks with two marked points. Conversely, a ciliated graph $\Gamma$ with set of edges $V$ is the skeleton of a marked surface $(\Sigma_\Gamma, V)$, well defined up to isomorphism, obtained by fusing marked disks corresponding to the edges of $\Gamma$. Thus the map $\Gamma \mapsto (\Sigma_\Gamma, V)$ gives a~bijective correspondence between isomorphism classes of ciliated graphs up to local changes in~\eqref{loc-move} and isomorphism classes of marked surfaces.

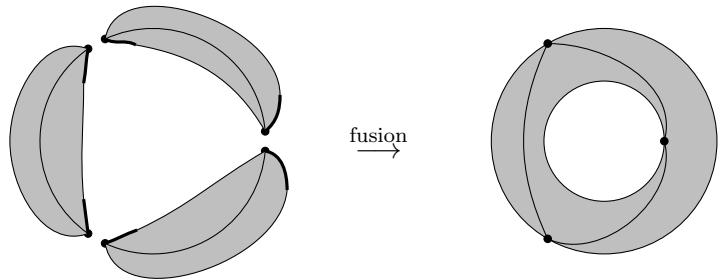
\begin{figure}[h]\centering
\begin{tikzpicture}
\pgfmathsetmacro{\ra}{1.5}
\pgfmathsetmacro{\rb}{0.8}
\begin{scope}[shift={(-3,0)}]
\coordinate (v0) at (5: \ra); \coordinate (v1) at (115: \ra); \coordinate (v2) at (125: \ra); \coordinate (v3) at (235: \ra); \coordinate (v4) at (245: \ra); \coordinate (v5) at (-5: \ra);
\draw [fill=black] (v0) circle [radius= 0.05];
\draw [fill=black] (v1) circle [radius= 0.05];
\draw [fill=black] (v2) circle [radius= 0.05];
\draw [fill=black] (v3) circle [radius= 0.05];
\draw [fill=black] (v4) circle [radius= 0.05];
\draw [fill=black] (v5) circle [radius= 0.05];
\draw [fill=lightgray](v0) to [out=40,in=-90] (20:\ra+ 0.3) to [out=90,in=80] (v1) to [out=-30,in=150] (100: \ra-0.2) to [out=-10,in=120] (v0);
\draw [very thick] (v0) to [out=40,in=-90] (20:\ra+ 0.3); \draw [very thick] (v1) to [out=-30,in=150] (100: \ra-0.2);
\draw [fill=lightgray] (v5) to [out=-10,in=90] (-20: \ra+0.4) to [out=-95,in=-80] (v4) to [out=20,in=200] (260: \ra-0.3) to [out=25,in=210] (v5);
\draw [very thick] (v5) to [out=-10,in=90] (-20: \ra+0.4); \draw [very thick] (v4) to [out=20,in=200] (260: \ra-0.3);
\draw [fill=lightgray] (v2) to [out=-100,in=80] (140: \ra -0.3) to [out=-90,in=95] (220:\ra-0.3) to (v3) to [out=200,in=-90] (180:\ra+0.4) to [out=90,in=160] (v2);
\draw [very thick] (v2) to [out=-100,in=80] (140: \ra -0.3); \draw [very thick] (220:\ra-0.3) to (v3);
\draw (v0) arc (5:115:\ra); \draw (v2) arc (125: 235:\ra); \draw (v4) arc (245:355:\ra);
\end{scope}
\node at (0,0) {$\stackrel{{\rm fusion}}{\longrightarrow}$};
\begin{scope}[shift={(3,0)}]
\draw [fill=lightgray] (0,0) circle [radius=\ra];
\draw [fill=white] (0,0) circle [radius=\rb];
\draw [fill=black] (0:\rb) circle [radius= 0.05];
\draw [fill=black] (120:\ra) circle [radius= 0.05];
\draw [fill=black] (240:\ra) circle [radius= 0.05];
\draw (0:\rb) to [out=75,in=0] (120:\ra) to [out=-115,in=115] (240:\ra) to [out=-25,in=-75] (0:\rb);
\end{scope}
\end{tikzpicture}
\caption{An annulus with three marked points obtained by fusing three disks with two marked points each.}
\end{figure}

\section{The Fock--Rosly Poisson structure} \label{sec-FR-Poiss}

In this section, we introduce a Poisson structure, f\/irst discovered by Fock and Rosly, on the moduli space of f\/lat connections over a marked surface, which is def\/ined by an action of a Lie algebra and a quasitriangular $r$-matrix.

Throughout Section~\ref{sec-FR-Poiss}, $G$ is a connected complex Lie group, and $\g$ is its Lie algebra.

\subsection{The moduli space of f\/lat connections over a marked surface} \label{subsec-moduli}

For a marked surface $(\Sigma, V)$, let $\Pi_1(\Sigma, V)$ be the fundamental groupoid of $\Sigma$ over the set of base points $V$, and consider
\begin{gather*}
\A(\Sigma, V) = \Hom(\Pi_1(\Sigma, V), G),
\end{gather*}
the moduli space of f\/lat connections on $G$-principal bundles over $\Sigma$ which are trivialized over~$V$. The group $G^V$ naturally acts on the right on $\A(\Sigma, V)$ by gauge transformations. For $p \in \Pi_1(\Sigma, V)$, denote by $\operatorname{ev}_p\colon \A(\Sigma, V) \to G$ the evaluation at~$p$, by $\theta(p), \tau(p) \in V$ the respective source and target of~$p$, and if $g \in G^V$ and $v \in V$, recall from Section~\ref{nota} that $g_v$ is the $v$'th component of~$g$. The action of $G^V$ on $\A(\Sigma, V)$ is then given by
\begin{gather}
\rho_{\sV}\colon \ \A(\Sigma, V) \times G^V \to \A(\Sigma, V), \qquad \operatorname{ev}_p(\rho_{\sV}(y, g)) = g_{\theta(p)}^{-1} \operatorname{ev}_p(y) g_{\tau(p)}, \label{lambda}
\end{gather}
where $g \in G^V$, $y \in \A(\Sigma, V)$, and $p \in \Pi_1(\Sigma, V)$. Given a skeleton $\Gamma$ of $(\Sigma, V)$ and an orientation of each edge of $\Gamma$, $\Pi_1(\Sigma, V)$ is then the free groupoid generated by $\Gamma$, and thus one has a natural dif\/feomorphism
\begin{gather}
I_{\Gamma}\colon \ G^{\Gamma_1} \stackrel{\cong}{\longrightarrow} \A(\Sigma, V). \label{orient-edge}
\end{gather}
Now, choose a pair of distinct marked points $(v_1, v_2)$. The fusion map $(\Sigma, V) \to (\Sigma_{(v_1, v_2)}, V_{v_1= v_2})$ induces a map
\begin{gather}
\varphi_{(v_1, v_2)}\colon \ \A\big(\Sigma_{(v_1, v_2)}, V_{v_1= v_2}\big) \longrightarrow \A(\Sigma, V), \label{varphi_v_1,v_2}
\end{gather}
and by choosing a skeleton for $(\Sigma, V)$, one easily sees that $\varphi_{(v_1, v_2)}$ is a dif\/feomorphism. The next lemma is straightforward.

\begin{Lemma} \label{big-action}
Let $\diag_{(v_1, v_2)}\colon G^{V_{v_1=v_2}} \to G^V$ be the canonical embedding. Then for any $g \in G^{V_{v_1= v_2}}$, one has
\begin{gather*}
\varphi_{(v_1, v_2)}\big(\rho_{\sV_{v_1=v_2}}(y, g)\big) = \rho_{\sV} \big( \varphi_{v_1, v_2}(y), \diag_{(v_1, v_2)}(g) \big), \qquad y \in \A(\Sigma_{(v_1, v_2)}, V_{v_1= v_2}).
\end{gather*}
\end{Lemma}

Let $\Gamma$ be a skeleton of $(\Sigma, V)$, and for any $\gamma \in \Gamma_1$, let $(D_\gamma, V_\gamma)$ be a disk marked with two points, where $V_\gamma \subset \Gamma_{1/2}$ consists of the two half edges of $\gamma$, so that $\sqcup_{\gamma \in \Gamma_1} V_\gamma= \Gamma_{1/2}$. As $(\Sigma, V)$ is a fusion of $(\sqcup_{\gamma \in \Gamma_1} D_\gamma, \Gamma_{1/2})$, one has a dif\/feomorphism
\begin{gather*}
\varphi_{\sGam}\colon \ \A(\Sigma, V) \stackrel{\cong}{\longrightarrow} \A\big(\sqcup_{\gamma \in \Gamma_1} D_\gamma, \Gamma_{1/2}\big),
\end{gather*}
and thus an action of $G^{\Gamma_{1/2}}$ on $\A(\Sigma, V)$ via $\varphi_{\sGam}$. Choosing an orientation for each edge of $\Gamma$ and identifying $\A(\Sigma, V) \cong G^{\Gamma_1}$ as in~\eqref{orient-edge}, the action of $G^{\Gamma_{1/2}}$ on $G^{\Gamma_1}$ is given by
\begin{gather}
\sigma_{\sGam}\colon \ G^{\Gamma_1} \times G^{\Gamma_{1/2}} \to G^{\Gamma_1}, \nonumber\\
\sigma_{\sGam}(g, h)_\gamma = h_{\alpha_\gamma}^{-1} g_\gamma h_{\calpha_\gamma}, \qquad h \in G^{\Gamma_{1/2}}, \qquad g \in G^{\Gamma_1}, \qquad \gamma \in \Gamma_1, \label{sigma}
\end{gather}
where for $\gamma \in \Gamma_1$, $\alpha_\gamma \in \Gamma_{1/2}$ is the source half edge of $\gamma$.

\subsection{The Fock--Rosly Poisson structure} \label{subsec-FR}

Let $(\Sigma, V)$ be a marked surface and $\Gamma$ an oriented skeleton of $(\Sigma, V)$. From now till the end of Section~\ref{sec-FR-Poiss}, we f\/ix an
\begin{gather*}
s \in \big(S^2\g\big)^\g.
\end{gather*}
For every $v \in V$, let $\Lambda_v \in \wedge^2 \g$ be such that $r_v = s + \Lambda_v$ is a quasitriangular $r$-matrix. Identi\-fying~$\g^{\sGam_v}$ with~$\g^{|\sGam_v|}$ using the linear order on $\Gamma_v$, let $r_v^{(\varepsilon_v, \sGam_v)} \in \g^{\sGam_v} \otimes \g^{\sGam_v}$ be as in~\eqref{defn-r^(n)}, where $\varepsilon_v\colon \Gamma_v \to \{ -1, 1\}$ is def\/ined as
\begin{gather*}
\varepsilon_v(\alpha) = \begin{cases}
\hphantom{-}1, & \alpha \ \text{is a source half edge}, \\
-1, & \alpha \ \text{is an end half edge},
\end{cases} \qquad \alpha \in \Gamma_v.
\end{gather*}
Since $\Gamma_{1/2} = \bigsqcup_{v \in V} \Gamma_v$, one has $\g^{\sGam_{1/2}} = \bigoplus_{v \in V} \g^{\sGam_v}$, and recalling our notation in Section~\ref{nota}, let
\begin{gather}
r_{\sGam} = \sum_{v \in V} \big( r_v^{(\varepsilon_v, \sGam_v)} \big)_v \in \g^{\sGam_{1/2}} \otimes \g^{\sGam_{1/2}}, \label{defn-r_sgam}
\end{gather}
a quasitriangular structure for the Lie bialgebra
\begin{gather*}
\bigoplus_{v \in V} \big( \g^{\sGam_v}, \delta_{r_v}^{(\sGam_v)} \big), 
\end{gather*}
where $\delta_{r_v}^{(\sGam_v)}$ is as in~\eqref{delta_r^n}. Using the notational convention in Section~\ref{nota}, recall from \eqref{sigma} the right Lie algebra action $\sigma_{\sGam}\colon \g^{\sGam_{1/2}} \to \XX^1(G^{\Gamma_1})$.

\begin{Theorem} \label{main-thm}
The bivector field
\begin{gather*}
\pi_{\sGam} = \sigma_{\sGam}(r_{\sGam}) \in \XX^2\big(G^{\Gamma_1}\big) 
\end{gather*}
is a Poisson structure on $G^{\Gamma_1}$.
\end{Theorem}

\begin{proof}
The symmetric part of $r_{\sGam}$ is
\begin{gather*}
s_{\sGam} = \sum_{\gamma \in \Gamma_1} (s)_{\alpha_\gamma} - (s)_{\calpha_\gamma},
\end{gather*}
where for $\gamma \in \Gamma_1$, $\alpha_\gamma \in \Gamma_{1/2}$ is the source half edge of $\gamma$. By \eqref{sigma}, one has $\sigma_{\sGam} (s_{\sGam}) = 0$, thus Theorem \ref{main-thm} follows from Proposition \ref{Poi-def-r}.
\end{proof}

\begin{Remark}Let $\Lambda_{\sGam}$ be the anti-symmetric part of the quasitriangular $r$-matrix $r_\sGam$. The bivector f\/ield
\begin{gather*}
\pi_\sGam = \sigma_\sGam(\Lambda_\sGam)
\end{gather*}
f\/irst appeared in \cite{FR2, FR1}, where the proof that it is Poisson was left as a computation to be checked. Theorem \ref{main-thm} gives a~simpler and more conceptual proof that $\pi_\Gamma$ is a Poisson structure.
\end{Remark}

Consider the quasitriangular Lie bialgebra
\begin{gather}
\big(\g^{\sV}, r\big) = \bigoplus_{v \in V} (\g, r_v), \qquad \text{where} \qquad r = \sum_{v \in V} (r_v)_v \in \g^{\sV} \otimes \g^{\sV}. \label{small-lie-bialg}
\end{gather}
For $v \in V$, let $\diag_v\colon \g \to \g^{\sGam_v}$, $\diag_v(x) = \sum_{\gamma \in \Gamma_v} (x)_\gamma$, for $x \in \g$, and def\/ine the map
\begin{gather}
\diag_{\sGam}\colon \ \g^{\sV} \to \g^{\sGam_{1/2}} = \bigoplus_{v \in V} \g^{\sGam_v}, \qquad \diag_{\sGam}(x) = \sum_{v \in V} (\diag_v(x_v))_v, \qquad x \in \g^{\sV}. \label{defn-Delta}
\end{gather}
By Theorem \ref{r^(n)}, $\diag_\sGam\colon (\g^{\sV}, \delta_r) \to (\g^{\sGam_{1/2}}, \delta_{r_{\sGam}})$ is an embedding of Lie bialgebras, and by Lemma~\ref{big-action}, one has
\begin{gather}
\rho_{\sV} = I_\Gamma \circ \sigma_{\sGam} \circ \diag_{\sGam}, \label{annoyingrelation}
\end{gather}
where $\rho_{\sV}\colon \g^\sV \to \XX^1(\A(\Sigma, V))$ is the derivative of the action by gauge transformations in \eqref{lambda}. Thus as an immediate consequence, one has

\begin{Corollary} \label{cor-gauge-action}
The triple $(\A(\Sigma, V), I_\Gamma(\pi_{\sGam}), \rho_{\sV})$ is a right $(\g^{\sV}, r)$-Poisson space.
\end{Corollary}

\begin{Example}
For any integer $n \geq 1$, let $(\Sigma_n, V)$ be a disk with $n-1$ inner disks removed and two marked points $v_1$, $v_2$, such as in Fig.~\ref{sigma_3:fig}, and let $\Gamma$ be the skeleton in Fig.~\ref{sigma_3:fig} with edges oriented from $v_1$ to $v_2$. Let $r_i = s + \Lambda_i$, be the $r$-matrix associated to $v_i$, $i=1,2$. The orientation of $\Sigma$ induces an isomorphism $\g^{\sGam_{v_i}} \cong \g^n$, and one has
\begin{gather*}
r_1^{(\varepsilon_{v_1}, \sGam_{v_1})} = (s, \ldots, s) + (\Lambda_1, \ldots, \Lambda_1) - \Mix^n(r_1) = (s, \ldots, s) + \Lambda_{r_1}^{(n)} \\
r_2^{(\varepsilon_{v_2}, \sGam_{v_2})} = (-s, \ldots, -s) + (\Lambda_2, \ldots, \Lambda_2) - \Mix^n(r_2) = (-s, \ldots, -s) + \Lambda_{r_2}^{(n)},
\end{gather*}
and
\begin{gather*}
I_{\Gamma}(\pi_{\sGam}) = \big(\Lambda_{r_1}^{(n)}\big)^R + \big(\Lambda_{r_2}^{(n)}\big)^L.
\end{gather*}
In particular, if $\Lambda_2 = -\Lambda_1$, the Poisson structure $I_{\Gamma}(\pi_{\sGam}) = \big(\Lambda_{r_1}^{(n)}\big)^R - \big(\Lambda_{r_1}^{(n)}\big)^L$ is multiplicative, and the Poisson Lie group
\begin{gather*}
\big(G^n, \big(\Lambda_{r_1}^{(n)}\big)^R - \big(\Lambda_{r_1}^{(n)}\big)^L\big)
\end{gather*}
is then called a {\it polyuble} in~\cite{FR2}, and has Lie bialgebra $\big(\g^n, \delta_{r_1}^{(n)}\big)$. When $n = 2$ and $r_1$ is {\it factorizable}, that is when $s^\sharp\colon \g^* \to \g$ is invertible, $\big(G^2, (\Lambda_{r_1}^{(2)})^R - (\Lambda_{r_1}^{(2)})^L\big)$ is isomorphic to the double of the Poisson Lie group $\big(G, \Lambda_1^R - \Lambda_1^L\big)$~\cite{Lu-Mou:mixed}. Thus polyubles are generalizations of doubles of Poisson Lie groups.

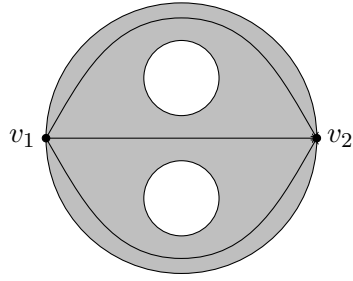
\begin{figure}\centering
\begin{tikzpicture}
\pgfmathsetmacro{\ra}{0.5}
\pgfmathsetmacro{\rb}{1.8}
\pgfmathsetmacro{\ry}{0.8}
\draw [fill=lightgray] (0,0) circle [radius=\rb];
\draw [fill=white] (0,\ry) circle [radius=\ra];
\draw [fill=white] (0,-\ry) circle [radius=\ra];
\draw [fill=black] (-\rb,0) circle [radius=0.05];
\draw [fill=black] (\rb, 0) circle [radius=0.05];
\node [left] at (-\rb,0) {$v_1$};
\node [right] at (\rb, 0) {$v_2$};
\draw [->] (180:\rb) to [out=60, in=180] (90:1.6) to [out=0, in=120] (0:\rb);
\draw [->] (180:\rb) to [out=300, in=180] (270:1.6) to [out=0, in=240] (0:\rb);
\draw [->] (180:\rb) to (\rb, 0);
\end{tikzpicture}
\caption{The marked surface $(\Sigma_3, V)$.} \label{sigma_3:fig}
\end{figure}
\end{Example}

\subsection{Independence of choice of skeleton} \label{subsec-indep}

Continuing with the setup and notation of Section~\ref{subsec-FR}, one has a Poisson structure $I_\Gamma(\pi_{\sGam})$ on~$\A(\Sigma,V)$. The goal of this subsection is to show that $I_\Gamma(\pi_{\sGam})$ does not depend on the choice of~$\Gamma$, nor on the choice of an orientation of the edges of~$\Gamma$. Letting $\Gamma'$ be another oriented skeleton of~$(\Sigma,V)$, this is equivalent to proving that the map
\begin{gather}
I_{\Gamma'}^{-1} \circ I_{\Gamma}\colon \ \big (G^{\Gamma_1}, \pi_{\sGam}\big) \longrightarrow \big(G^{\Gamma'_1}, \pi_{\sGam'}\big) \label{IcircI}
\end{gather}
is a Poisson isomorphism.

\begin{Lemma} \label{lem-orient}
The Poisson structure $I_\Gamma(\pi_{\sGam})$ is independent of the orientation of the edges of $\Gamma$.
\end{Lemma}
\begin{proof}
One can assume that $\Gamma'$ is the same oriented skeleton as $\Gamma$, except for an edge $\gamma \in \Gamma_1$, which is given the opposite orientation. The map $I_{\Gamma'}^{-1} \circ I_{\Gamma}\colon G^{\Gamma_1} \to G^{\Gamma'_1} = G^{\Gamma_1}$ is thus the group inversion in the factor $\gamma$, and the identity on all other factors, hence for any $x \in \g^{\sGam_{1/2}}$, one has
\begin{gather*}
I_{\Gamma'}^{-1} \circ I_{\Gamma}(\sigma_{\sGam}(x)) = \sigma_{\sGam'}(x),
\end{gather*}
which implies that $I_{\Gamma'}^{-1} \circ I_{\Gamma}(\pi_{\sGam}) = \pi_{\sGam'}$.
\end{proof}

\begin{Lemma} \label{lem-indep}
Consider the following two oriented skeletons
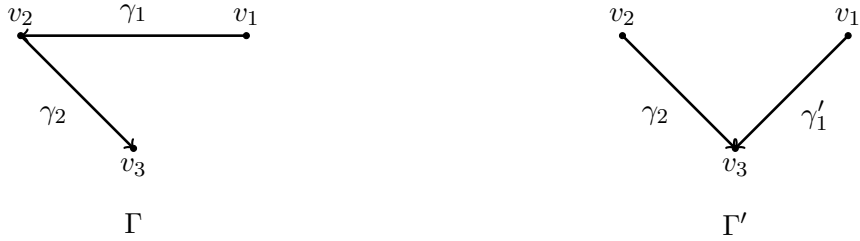
\begin{center}
\begin{tikzpicture}
\pgfmathsetmacro{\ra}{1.5}
\pgfmathsetmacro{\rc}{2.5}
\begin{scope}[shift={(-4,0)}]
\draw [fill=black] (0:\ra) circle [radius=0.04];
\draw [fill=black] (180: \ra) circle [radius=0.04];
\draw [fill=black] (270:\ra) circle [radius=0.04];
\draw [->, line width=1pt] (0:\ra) -- (180: \ra);
\draw [->, line width=1pt] (180: \ra) -- (270:\ra);
\node [above] at (0:\ra) {$v_1$};
\node [above] at (180:\ra) {$v_2$};
\node [below] at (270:\ra) {$v_3$};
\node at (90: 0.3) {$\gamma_1$};
\node at (225: 1.5) {$\gamma_2$};
\node at (270: \rc) {$\Gamma$};
\end{scope}
\begin{scope}[shift={(4,0)}]
\draw [fill=black] (0:\ra) circle [radius=0.04];
\draw [fill=black] (180: \ra) circle [radius=0.04];
\draw [fill=black] (270:\ra) circle [radius=0.04];
\draw [->, line width=1pt] (0:\ra) -- (270:\ra);
\draw [->, line width=1pt] (180: \ra) -- (270:\ra);
\node [above] at (0:\ra) {$v_1$};
\node [above] at (180:\ra) {$v_2$};
\node [below] at (270:\ra) {$v_3$};
\node at (315: 1.5) {$\gamma_1'$};
\node at (225: 1.5) {$\gamma_2$};
\node at (270: \rc) {$\Gamma'$};
\end{scope}
\end{tikzpicture}
\end{center}
of a disk with three marked points. Then the map~\eqref{IcircI} is a Poisson isomorphism.
\end{Lemma}
\begin{proof}Identifying $G^{\Gamma_1}$ and $G^{\Gamma_1'}$ with $G^2$ and writing $I = I_{\Gamma'}^{-1} \circ I_{\Gamma}$, one has $I(g_1, g_2) = (g_1g_2, g_2)$, $g_1, g_2 \in G$, and
\begin{gather*}
\pi_{\sGam} = \big(\Lambda_{v_1}^R, \Lambda_{v_3}^L\big) + \big(\Lambda_{v_2}^L, \Lambda_{v_2}^R\big) + \sum_i \big(0, y_i^R\big) \wedge \big(x_i^L, 0\big),
\end{gather*}
where $r_{v_2} = \sum_i x_i \otimes y_i$. A direct calculation using $\Ad_g(s) = s$ for any $g \in G$, shows that
\begin{gather*}
\big(\Lambda_{v_1}^R, 0\big) = I\big( \Lambda_{v_1}^R, 0\big), \\
\big(\Lambda_{v_3}^L, \Lambda_{v_3}^L\big) - \Mix^2(\Lambda_{v_3})^L = I\big(0, \Lambda_{v_3}^L\big), \\
\big(0, \Lambda_{v_2}^R\big) - \Mix^2(s)^L = I \bigg(\big(\Lambda_{v_2}^L, \Lambda_{v_2}^R\big)+ \sum_i \big(0, y_i^R\big) \wedge \big(x_i^L, 0\big) \bigg),
\end{gather*}
from which one gets $I(\pi_{\sGam}) = \big(\Lambda_{v_1}^R, \Lambda_{v_2}^R\big) + \big(\Lambda_{v_3}^L, \Lambda_{v_3}^L\big) - \Mix^2(r_{v_3})^L = \pi_{\sGam'}$.
\end{proof}

We return to the case of a general marked surface $(\Sigma, V)$.

\begin{Proposition} \label{indep-thm}Let $\Gamma, \Gamma'$ be oriented skeletons of $(\Sigma, V)$. Then the map \eqref{IcircI} is a Poisson isomorphism.
\end{Proposition}
\begin{proof}
By Lemma \ref{lem-orient} and Proposition \ref{ex-unique}, one can assume that $\Gamma$ and $\Gamma'$ have the following form:
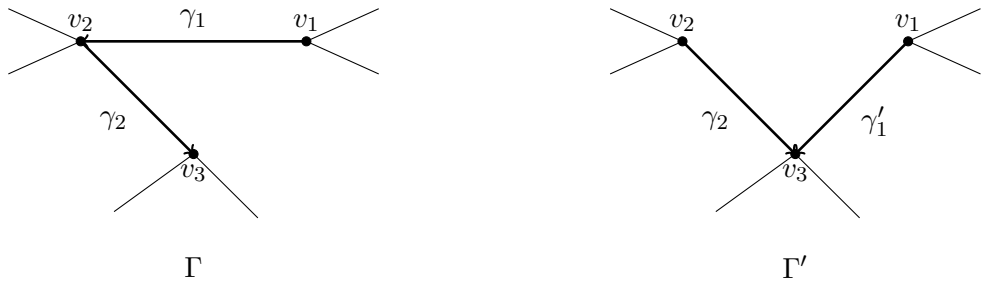
\begin{center}
\begin{tikzpicture}
\pgfmathsetmacro{\ra}{1.5}
\pgfmathsetmacro{\rb}{2.5}
\pgfmathsetmacro{\rc}{3}
\begin{scope}[shift={(-4,0)}]
\draw [fill=black] (0:\ra) circle [radius=0.06];
\draw [fill=black] (180: \ra) circle [radius=0.06];
\draw [fill=black] (270:\ra) circle [radius=0.06];
\draw [->, line width=1pt] (0:\ra) -- (180: \ra);
\draw [->, line width=1pt] (180: \ra) -- (270:\ra);
\draw (180: \ra) -- (170:\rb);
\draw (180: \ra) -- (190:\rb);
\draw (270:\ra) --(245:\rb);
\draw (270:\ra) --(290:\rb);
\draw (0:\ra) -- (10:\rb);
\draw (0:\ra) -- (350:\rb);
\node [above] at (0:\ra) {$v_1$};
\node [above] at (180:\ra) {$v_2$};
\node [below] at (270:\ra) {$v_3$};
\node at (90: 0.3) {$\gamma_1$};
\node at (225: 1.5) {$\gamma_2$};
\node at (270: \rc) {$\Gamma$};
\end{scope}
\begin{scope}[shift={(4,0)}]
\draw [fill=black] (0:\ra) circle [radius=0.06];
\draw [fill=black] (180: \ra) circle [radius=0.06];
\draw [fill=black] (270:\ra) circle [radius=0.06];
\draw [->, line width=1pt] (0:\ra) -- (270:\ra);
\draw [->, line width=1pt] (180: \ra) -- (270:\ra);
\draw (180: \ra) -- (170:\rb);
\draw (180: \ra) -- (190:\rb);
\draw (270:\ra) --(245:\rb);
\draw (270:\ra) --(290:\rb);
\draw (0:\ra) -- (10:\rb);
\draw (0:\ra) -- (350:\rb);
\node [above] at (0:\ra) {$v_1$};
\node [above] at (180:\ra) {$v_2$};
\node [below] at (270:\ra) {$v_3$};
\node at (315: 1.5) {$\gamma_1'$};
\node at (225: 1.5) {$\gamma_2$};
\node at (270: \rc) {$\Gamma'$};
\end{scope}
\end{tikzpicture}
\end{center}
Using Lemma \ref{lem-indep}, a straightforward calculation, the details of which are left to the readers, shows that $I_{\Gamma'}^{-1} \circ I_{\Gamma}(\pi_{\sGam}) = \pi_{\sGam'}$.
\end{proof}

Recall the quasitriangular $r$-matrix $r$ on $\g^{\sV}$ def\/ined in \eqref{small-lie-bialg} and let
\begin{gather*}
\pi_r = I_\Gamma(\pi_{\sGam}) \in \XX^2(\A(\Sigma, V)),
\end{gather*}
where $\Gamma$ is any oriented skeleton of $(\Sigma, V)$. The following theorem is a consequence of Lemma~\ref{lem-orient}, Proposition~\ref{indep-thm}, and Lemma~\ref{cor-gauge-action}.

\begin{Theorem} \label{main-thm-2}The Poisson structure $\pi_r$ on $\A(\Sigma, V)$ is independent of the choice of $\Gamma$, and $(\A(\Sigma, V), \pi_r, \rho_{\sV})$ is a right $(\g^{\sV}, r)$-Poisson space.
\end{Theorem}

\subsection{Fusion of Poisson spaces and marked surfaces} \label{subsec-fus}

We continue with the setup of Section~\ref{subsec-FR}. In particular, $(\Sigma, V)$ is a marked surface, for $v \in V$, one has $\Lambda_v \in \wedge^2 \g$ such that $r_v = s + \Lambda_v$ is a quasitriangular $r$-matrix, and one considers the quasitriangular $r$-matrix $r \in \g^{\sV} \otimes \g^{\sV}$ def\/ined in~\eqref{small-lie-bialg}.

Suppose one has $r_{v_1} = r_{v_2}$ for two distinct elements $v_1, v_2 \in V$, and consider the fused surface $(\Sigma', V') = (\Sigma_{(v_1, v_2)}, V_{v_1 = v_2})$ with Poisson structure $\pi_{r'}$, where $r' \in \g^{\sV'} \otimes \g^{\sV'}$ is def\/ined as in~\eqref{small-lie-bialg}, and let $v' \in V'$ be the vertex obtained by fusing $v_1$ and $v_2$. Recall the fusion of Poisson spaces discussed in Section~\ref{subsec-fusion}.

\begin{Theorem} \label{thm-fuse-mark}One has
\begin{gather*}
\Fus_{(\g, r_{v_1}) \times (\g, r_{v_2})} (\A(\Sigma, V), \pi_r, \rho_\sV) = \big(\A(\Sigma', V'), \pi_{r'}, \rho_{\sV'}\big).
\end{gather*}
\end{Theorem}
\begin{proof}We identify $\A(\Sigma, V)$ and $\A(\Sigma', V')$ using the natural dif\/feomorphism $\varphi_{(v_1, v_2)}$ in~\eqref{varphi_v_1,v_2}, and let $(\A(\Sigma', V'), \pi', \rho')$ be the fusion at $(\g, r_{v_1}) \times (\g, r_{v_2})$ of $(\A(\Sigma, V), \pi_r, \rho_\sV)$. Let $\Gamma$ be an oriented skeleton of $(\Sigma, V)$ and recall from Lemma \ref{fusion-graph-surf} that $\Gamma' = \Gamma_{(v_1, v_2)}$ is an oriented skeleton of $(\Sigma', V')$, and that one has a natural identif\/ication $\Gamma_1 \cong \Gamma'_1$. Recall the map $\diag_{\sGam}\colon \g^{\sV} \to \g^{\sGam_{1/2}}$ def\/ined in \eqref{defn-Delta} and consider its restriction $\diag_{\sGam} |_{\g^{\{v_1, v_2\}}}$ to $\g^{\{v_1, v_2\}}$. Using~\eqref{annoyingrelation} and Lemma~\ref{lem-mix^2}, one has
\begin{gather*}
\pi' = \pi_r - \rho_\sV |_{\g^{\{v_1, v_2\}}}\big(\Mix^2(r_{v_1})\big)
= I_\Gamma \circ \sigma_{\sGam}\big( r_{\sGam} - \diag_{\sGam} |_{\g^{\{v_1, v_2\}}} \big(\Mix^2(r_{v_1})\big)\big)\\
\hphantom{\pi'}{} = I_{\Gamma'} \circ \sigma_{\sGam'}(r_{\sGam'}) = \pi_{r'},
\end{gather*}
and by Lemma \ref{big-action}, one has $\rho' = \rho_{\sV'}$.
\end{proof}

\begin{Example} \label{exa-fused-disk}
Let $(\Sigma, V)$ be a disk with two marked points $v_1$, $v_2$ and assume that the $r$-matrices $r_1$ and $r_2$ associated to~$v_1$ and~$v_2$ are equal. Let the edge of $\Gamma$ be oriented from~$v_1$ to~$v_2$ and identify $\A(\Sigma, V) \cong G$ via $I_\Gamma$, so that one has
\begin{gather*}
\pi_r = \Lambda^R + \Lambda^L,
\end{gather*}
where $\Lambda$ is the anti-symmetric part of $r_1=r_2$.
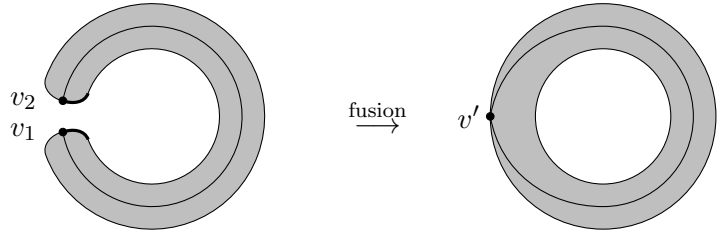
\begin{center}
\begin{tikzpicture}
\pgfmathsetmacro{\ra}{1.5}
\pgfmathsetmacro{\rb}{1.2}
\pgfmathsetmacro{\rc}{0.9}
\begin{scope}[shift={(-3,0)}]
\draw [fill=lightgray] (160:\ra) arc (160: -160:\ra) to [out=-250,in=-170] (-170:\rb) to [out=20,in=-250] (-160:\rc) arc (-160:160:\rc) to [out=250,in=-20] (170:\rb) to [out=170,in=250] (160:\ra);
\draw [very thick] (-170:\rb) to [out=20,in=-250] (-160:\rc);
\draw [very thick] (160:\rc) to [out=250,in=-20] (170:\rb);
\draw (170:\rb) arc (170: -170:\rb);
\draw [fill=black] (170:\rb) circle [radius=0.05];
\draw [fill=black] (-170:\rb) circle [radius=0.05];
\node [left] at (170:\rb) {$v_2\;\;$};
\node [left] at (-170:\rb) {$v_1\;\;$};
\end{scope}
\node at (0,0) {$\stackrel{{\rm fusion}}{\longrightarrow}$};
\begin{scope}[shift={(3,0)}]
\draw [fill=lightgray] (0,0) circle [radius=\ra];
\draw [fill=white] (0,0) circle [radius=\rc];
\draw [fill=black] (180:\ra) circle [radius=0.05];
\draw (180:\ra) to [out=80,in=180] (90:\rb) arc (90:-90:\rb) to [out=180,in=-80] (180:\ra);
\node [left] at (180:\ra) {$v'$};
\end{scope}
\end{tikzpicture}
\end{center}
The fused surface $(\Sigma', \{v'\})$ is an annulus with one marked point, and one has
\begin{gather*}
\pi_{r'} = \Lambda^R + \Lambda^L - \rho_{\sV} \big(\Mix^2(r_1)\big) = \Lambda^R + \Lambda^L + \sum_i y_i^R \wedge x_i^L \\
 \hphantom{\pi_{r'}}{} = \sum_i \frac{1}{2} x_i^R \wedge y_i^R + \frac{1}{2} x_i^L \wedge y_i^L + y_i^R \wedge x_i^L,
\end{gather*}
where $r_1= r_2 = \sum_i x_i \otimes y_i$.
\end{Example}

\section{Quasi Poisson geometry} \label{sec-quasipoisson}

\subsection{Quasi Poisson spaces and fusion of quasi Poisson spaces} \label{subsec-quasi}

Let $\g$ be a Lie algebra, $s \in (S^2\g)^\g$, and recall the element $\phi_s \in (\wedge^3 \g)^\g$ def\/ined in~\eqref{eq-cyb}. Recall from \cite{AA} that a {\it right $(\g, \phi_s)$-quasi Poisson space} is a triple $(Y, Q_Y, \rho)$, where $Y$ is a manifold, $\rho\colon \g \to \XX^1(Y)$ a right Lie algebra action, and $Q_Y \in \XX^2(Y)$ is a $\g$-invariant bivector f\/ield on~$Y$, such that
\begin{gather*}
[Q_Y, Q_Y] = \rho(\phi_s).
\end{gather*}
We denote by $\Q\P(\g, \phi_s)$ the category of right $(\g, \phi_s)$-quasi Poisson spaces, where the morphisms are $\g$-equivariant smooth maps respecting the quasi Poisson bivector f\/ields, and if $r \in \g \otimes \g$ is a quasitriangular $r$-matrix on~$\g$, denote by $\P(\g, r)$ the category of right $(\g, r)$-Poisson spaces, where the morphisms are $\g$-equivariant Poisson maps.

\begin{Proposition}[\cite{Anton-Yvette, David-Severa:quasi-Hamiltonian-groupoids, Lu-Mou:mixed}]\label{lem-equiv-quasi}
Let $\Lambda \in \wedge^2 \g$ such that $r = s + \Lambda$ is a quasitriangular $r$-matrix. Then one has an equivalence of categories
\begin{gather*}
\P(\g, r) \longleftrightarrow \Q\P(\g, \phi_s), \qquad (Y, \piY, \rho) \mapsto (Y, \piY - \rho(\Lambda), \rho), 
\end{gather*}
where the functor on the level of morphisms is the identity map.
\end{Proposition}
\begin{proof}
Let $(Y, \piY, \rho) $ be a right $(\g, r)$-Poisson space. Then
\begin{gather*}
[\piY - \rho(\Lambda), \piY - \rho(\Lambda)] = - 2[\rho(\Lambda), \piY] + \rho([\Lambda, \Lambda]) = -\rho([\Lambda, \Lambda]) = \rho(\phi_s),
\end{gather*}
and for $x \in \g$,
\begin{gather*}
[\rho(x), \piY - \rho(\Lambda)] = \rho(\delta_r(x) - [x, \Lambda]) = 0,
\end{gather*}
hence $(Y, \piY - \rho(\Lambda), \rho)$ is a right $(\g, \phi_s)$-quasi Poisson space. One similarly checks that if $(Y, Q_Y, \rho)$ is a $(\g, \phi_s)$-quasi Poisson space, $(Y, Q_Y + \rho(\Lambda), \rho)$ is a right $(\g, r)$-Poisson space.
\end{proof}

\subsection[A canonical quasi Poisson structure on $\A(\Sigma, V)$]{A canonical quasi Poisson structure on $\boldsymbol{\A(\Sigma, V)}$} \label{subsec-quasip}

Let $G$ be a connected complex Lie group, $\g$ its Lie algebra, and let $s \in (S^2\g)^\g$. Let~$(\Sigma, V)$ be a~marked surface and for $v \in V$, let $\Lambda_v \in \wedge^2 \g$ be such that $r_v = s + \Lambda_v$ is quasitriangular $r$-matrix, and let $r \in \g^{\sV} \otimes \g^{\sV}$ be as in~\eqref{small-lie-bialg}. By Proposition~\ref{lem-equiv-quasi},
\begin{gather*}
\bigg(\A(\Sigma, V), \, Q_s: = \pi_r - \rho_{\sV} \bigg( \sum_{v \in V} (\Lambda_v)_v \bigg), \, \rho_{\sV} \bigg)
\end{gather*}
is a right $(\g, \phi_s)^{\sV}$-quasi Poisson space.

\begin{Proposition}[see also \cite{LBS1}] \label{coord-quasi}
For any oriented skeleton $\Gamma$ of $(\Sigma, V)$, one has
\begin{gather}
I_\Gamma^{-1}(Q_s) = -\sigma_{\sGam} \bigg( \sum_{v \in V} \big(\Mix^{\sGam_v}(s)\big)_v \bigg), \label{form-coord-quasi}
\end{gather}
where for $v \in V$, $\Mix^{\sGam_v}(s) \in \wedge^2 (\g^{\sGam_v}) \cong \wedge^2 (\g^{|\sGam_v |})$ is the element defined in \eqref{Mix^n(r)} using the linear order of $\Gamma_v$. In particular, $Q_s$ depends only on $s \in (S^2\g)^\g$.
\end{Proposition}
\begin{proof} Let $\Lambda_\sGam \in \wedge^2 \g^{\sGam_{1/2}}$ be the anti-symmetric part of the quasitriangular $r$-matrix $r_\sGam$ in~\eqref{defn-r_sgam} and recall the map $\diag_\sGam\colon \g^\sV \to \g^{\sGam_{1/2}}$ in~\eqref{defn-Delta}. By Lemma~\ref{lem-diag_nLambda}, one has
\begin{gather*}
I_\Gamma^{-1}(Q_s) = \sigma_{\sGam} \bigg( \Lambda_\sGam - \diag_{\sGam} \bigg(\sum_{v \in V} (\Lambda_v)_v \bigg) \bigg) = -\sigma_{\sGam} \bigg( \sum_{v \in V} (\Mix^{\sGam_v}(s))_v \bigg).\tag*{\qed}
\end{gather*}\renewcommand{\qed}{}
\end{proof}

\begin{Remark} Formula \eqref{form-coord-quasi} appears in \cite[Section~4]{LBS1}, where $Q_s$ was def\/ined via fusion of quasi-Poisson spaces.

Note that both $\pi_r$ and $Q_s$ descend to the same Poisson structure on the moduli space $\A(\Sigma, V)/G^V$ of f\/lat $G$-connections over $\Sigma$.
\end{Remark}

\begin{Example}Let $(\Sigma', \{v'\})$ be the annulus with one marked point in Example \ref{exa-fused-disk}, and let the edge of its skeleton $\Gamma'$ be oriented in the anti-clockwise direction. Identify $\A(\Sigma', \{v'\}) \cong G$ via $I_{\Gamma'}$ and $\g^{\Gamma'_{1/2}} \cong \g^2$. Writing $r = \sum_i x_i \otimes y_i$, one has $2s = \sum_i (x_i \otimes y_i + y_i \otimes x_i)$, thus
\begin{align*}
Q_s & =- \sigma_{\sGam'}\big(\Mix^2(s)\big) = \frac{1}{2} \sum_i x_i^R \wedge y_i^L + y_i^R \wedge y_i^L.
\end{align*}
Let $\rho$ be the (right) action of $G$ on itself by conjugation. The right $(\g, \phi_s)$-quasi Poisson space $(G, Q_s, \rho)$ is an example of the Hamiltonian quasi-Poisson spaces which were studied in \cite{AA}.
\end{Example}

\subsection*{Acknowledgements}

The author wishes to thank Jiang-Hua Lu, Marco Gualtieri and Francis Bischof\/f for their helpful comments. The author also wishes to thank the anonymous referees and editors, whose suggestions helped improve this paper.

\pdfbookmark[1]{References}{ref}
\LastPageEnding


\begin{thebibliography}{99}
\footnotesize\itemsep=0pt

\bibitem{Anton-Yvette}
Alekseev A., Kosmann-Schwarzbach Y., Manin pairs and moment maps,
 \href{https://doi.org/10.4310/jdg/1090347528}{\textit{J.~Differential Geom.}} \textbf{56} (2000), 133--165,
 \href{https://arxiv.org/abs/math.DG/9909176}{math.DG/9909176}.

\bibitem{AA}
Alekseev A., Kosmann-Schwarzbach Y., Meinrenken E., Quasi-{P}oisson manifolds,
 \href{https://doi.org/10.4153/CJM-2002-001-5}{\textit{Canad.~J. Math.}} \textbf{54} (2002), 3--29, \href{https://arxiv.org/abs/math.DG/0006168}{math.DG/0006168}.

\bibitem{AB}
Atiyah M.F., Bott R., The {Y}ang--{M}ills equations over {R}iemann surfaces,
 \href{https://doi.org/10.1098/rsta.1983.0017}{\textit{Philos. Trans. Roy. Soc. London Ser.~A}} \textbf{308} (1983),
 523--615.

\bibitem{FR2}
Fock V.V., Rosly A.A., Flat connections and polyubles, \href{https://doi.org/10.1007/BF01017138}{\textit{Theoret. and
 Math. Phys.}} \textbf{95} (1993), 228--238.

\bibitem{FR1}
Fock V.V., Rosly A.A., Poisson structure on moduli of f\/lat connections on
 {R}iemann surfaces and the {$r$}-mat\-rix, in Moscow {S}eminar in
 {M}athematical {P}hysics, \href{https://doi.org/10.1090/trans2/191/03}{\textit{Amer. Math. Soc. Transl. Ser.~2}}, Vol.~191,
 Amer. Math. Soc., Providence, RI, 1999, 67--86, \href{https://arxiv.org/abs/math.QA/9802054}{math.QA/9802054}.

\bibitem{David-Severa:quasi-Hamiltonian-groupoids}
Li-Bland D., \v{S}evera P., Quasi-{H}amiltonian groupoids and multiplicative
 {M}anin pairs, \href{https://doi.org/10.1093/imrn/rnq170}{\textit{Int. Math. Res. Not.}} \textbf{2011} (2011),
 2295--2350, \href{https://arxiv.org/abs/0911.2179}{arXiv:0911.2179}.

\bibitem{LBS1}
Li-Bland D., \v{S}evera P., Moduli spaces for quilted surfaces and {P}oisson
 structures, \textit{Doc. Math.} \textbf{20} (2015), 1071--1135,
 \href{https://arxiv.org/abs/1212.2097}{arXiv:1212.2097}.

\bibitem{Lu-Mou:mixed}
Lu J.-H., Mouquin V., Mixed product Poisson structures associated to {P}oisson
 {L}ie groups and {L}ie bialgebras, \href{https://doi.org/10.1093/imrn/rnw189}{\textit{Int. Math. Res. Not.}}, {t}o appear,
 \href{https://arxiv.org/abs/1504.06843}{arXiv:1504.06843}.

\bibitem{mass-turaev}
Massuyeau G., Turaev V., Quasi-{P}oisson structures on representation spaces of
 surfaces, \href{https://doi.org/10.1093/imrn/rns215}{\textit{Int. Math. Res. Not.}} \textbf{2014} (2014), 1--64,
 \href{https://arxiv.org/abs/1205.4898}{arXiv:1205.4898}.

\bibitem{physics1}
Meusburger C., Wise D.K., Hopf algebra gauge theory on a ribbon graph,
 \href{https://arxiv.org/abs/1512.03966}{arXiv:1512.03966}.

\end{thebibliography}
\end{document}